\documentclass[11pt]{article}
\usepackage{amsmath,amstext,amssymb,amsthm}

\usepackage[english]{babel}
\usepackage[T1]{fontenc}
\usepackage[utf8]{inputenc}
\usepackage{tikz}
\usepackage{url}
\usepackage{times}

\usepackage{hyperref}

\newtheorem{theorem}{Theorem}
\newtheorem{lemma}[theorem]{Lemma}

\theoremstyle{definition}
\newtheorem{example}{Example}

\newtheorem{problem}{Problem}

\newcommand{\Hall}{\mathbf{m}}
\newcommand{\Al}{\Sigma}

\newcommand{\eps}{\varepsilon}
\newcommand{\per}{\partial}

\newcommand{\EXTRA}[1]{}

\newcommand{\Cr}{\eta}
\newcommand{\sdot}{{\cdot}}
\newcommand{\tm}{\mathbf{t}}

\begin{document}

\title{A Note on Squares in Binary Words}

\author{Tero Harju\\
Department of Mathematics and Statistics\\
         University of~Turku, Finland\\
         \texttt{harju@utu.fi}\\
         }

\maketitle

\begin{abstract}
We consider words over a binary alphabet.
A word $w$ is overlap-free if it does not have factors 
(blocks of consecutive letters) of the form $uvuvu$ for
nonempty $u$. Let $M(w)$ denote the number of positions
that are middle positions of squares in $w$.
We show that for overlap-free binary words,
$2M(w) \le |w|+3$, and that there are infinitely many 
overlap-free binary words for which $2M(w)=|w|+3$.
\end{abstract}


\section{Introduction}

Overlap-free words were first studied by Axel Thue in two papers~\cite{Thue:06,Thue:12} in 1906 and 1912. In particular
he proved that there are infinite overlap-free binary words.

In the following we consider binary words, and choose the alphabet
$\Al=\{0,1\}$ for these words.
An \emph{overlap} is a word of the form $uvuvu$ where 
$u$ is nonempty. 
A word $w$ is said to be \emph{overlap-free} if it has no 
overlaps as factors, i.e., consecutive blocks of letters.
Equally well, we can presume in the definition that $x$ is a letter. 
The most well known overlap-free words are the factors of the 
infinite \emph{Thue-Morse word} $\tm$ that is obtained by
iterating the morphism 
$\mu\colon \Al^* \to \Al^*$ with
\[
\mu(0)=01 \ \text{  and } \ \mu(1)=10
\]
on the initial word $0$. Therefore
\[
\tm= 0110 1001 1001 0110 1001 0110 \cdots
\]
The morphism $\mu$ preserves overlap-freeness,
i.e., if $w$ is overlap-free, then so is $\mu(w)$; see e.g.
Chapter~2 of Lothaire~\cite{LothaireI} or Berstel and S\'e\'ebold~\cite{BerstelSeebold}.

Let $|w|$ denote the length of $w$.
An integer~$p$ with \hbox{$1\leq p < |w|$} is  
a \emph{position} in~$w$. It denotes the place after the prefix 
of length $p$.
A nonempty word $u$ is a \emph{square at
position}~$p=|x|$  in $w=xy$ if there are, possibly empty, words $x'$ and $y'$ such that $x=x'u$ and $y=uy'$.
In this case we say that $p$ is a \emph{centre of a square}.

Let $\eps$ denote the empty word.
If $w=xuy$ then $u$ is a~\emph{factor} of $w$.
It is a \emph{prefix} if $x=\eps$, and a \emph{suffix} 
if $y=\eps$.
The word~$w$ is said to be \emph{bordered} if there exists a nonempty word~$v$, with $v \ne w$, that is both a prefix and a suffix of $w$. 
For a word $w=uv$, we denote by $u=wv^{-1}$  the prefix of $w$
when the suffix $v$ is deleted.

\section{On the maximum number of centres}

We consider the number of centres of squares in binary words.
We follow Harju and K\"arki~\cite{HarjuKarki}, and define
\[
M(w) = \# \{p \mid \text{$p$ is a centre of a square in $w$} \}.
\]
In~\cite{HarjuKarki} the authors count `frames', i.e.,
unbordered squares. The number of these equals $M(w)$
since if $uu$ is a square
at position $p$ then for the border $v$ of $u$ of minimal length 
also $vv$ is a square at $p$.

In the general binary case the paper~\cite[Theorem~4]{HarjuKarki}
gives the minimum number of centres of squares:

\begin{theorem}
For binary words of length $n \ge 3$, 	
\[
\min_{|w|=n}  M(w) = \left\lceil \frac{n}{2}\right\rceil-2.
\]
\end{theorem}

The problem for $\max_{|w|=n} M(w)$ is trivial. Indeed, the unary word $0^n$ has a square at every position,
i.e., $M(0^n)=n-1$. On the other hand, in the ternary case,
i.e., words over a three letter alphabet, the minimum is zero
since there are (infinite) square-free words in this case.
Here we shall consider the maximum problem for the 
binary overlap-free words.

The next technical lemma will be
used for short prefixes $x$ in the proof of Theorem~\ref{main}.

\begin{lemma}\label{squares}
Let $s$ be the longest common suffix of the binary words $x$ and $w$.
Assume that $xw$ and $ww$ are
overlap-free. Then $xw(ws^{-1})$ is overlap-free.
\end{lemma}

\begin{proof}
Assume that there is a square $xuxu$ as a prefix of $xww$ such that
$|uxu| \ge |w|$.
Then $u$ is a prefix $w$, i.e.,  $w=uy=vt$ for $y$ and $t$
with $xu=yv$;
see Fig.~\ref{Fig:1}.

The word $u$ is a suffix of $yv \ (=xu)$ and a prefix of $w=vt \ (=uy)$.
The occurrences of $u$ cannot overlap in the overlap-free $ww$.
Therefore either
(1) $u=v$ or (2) $v=uzu$ for some $z$.

\smallskip
(1) Suppose first that $u=v$ and thus that $y=x$ and $t=x$. 
Then $x$ is a suffix of the common suffix $s$
and $xw(wx^{-1})=xuxu$. Therefore 
$xuxu$ does not extend to an overlapping factor $xuxua$ of $xw(ws^{-1})$.

\smallskip
 (2) Assume then that $v=uzu$ and hence that $y=zut$ and $w=uzut$. 
Then $xu=yv=zutuzu$, and finally,
$xw=xuy=zutuzuzut$ which contains an overlap $uzuzu$
(since $u \ne \eps$); a contradiction with the overlap-freeness of $xw$.
\end{proof}

\begin{figure}[htb]
\tikzstyle{mystate}=[inner sep=0pt, outer sep=0pt]
\begin{center}
\begin{tikzpicture}[line width=0.06pc, scale=1](10,2)
\node  [mystate] (Z) at (-1,0.5) {};
\node  [mystate] (A) at (1,0.5) {} ; 
\node  [mystate] (B) at (10,0.5) {};
\node [mystate]  (D) at (5,0.5) {};
\draw (A)--(B);
\draw (Z)--(A);
\draw (Z)--(-1,0.7);
\draw (A)--(1,0.80);
\draw (A)--(1,0.20);
\draw (B)--(10,0.8);
\draw (D)--(5,0.8);
\draw (D)--(5,0.2);
\node at (3,0.8) {$w$};
\node at (7,0.8) {$w$};
\node at (0.0,0.8) {$x$};
\node (C) at (3,0.5){};
\draw (3,0.5)--(3,0.2);
\node (E) at (7,0.5){};
\draw (7,0.5)--(7,0.2);
\node at (2,0.04) {$u$};
\node at (4.0,0.04) {$y$};
\node at (6.0,0.04) {$v=uzu$};
\node at (8.25,0.07) {$t$};
 \end{tikzpicture}
 \end{center}
 \caption{Overlapping factor: $xu=yv$.
}
 \label{Fig:1}
\end{figure}
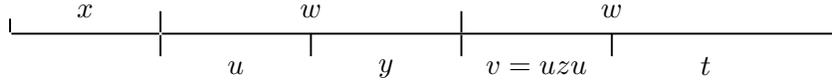

In particular, if in the above $s=\eps$ and $xw$ and $ww$ are overlap-free
then also $xww$ is overlap-free.

The next lemma was proved in~\cite[Lemma~10]{HarjuKarki}.
We note that in the Thue-Morse word~$\tm$ 
there are no consecutive centres of squares. This
can be seen by looking at the short factors of $\tm$. Indeed,
the unbordered squares in $\tm$ are of the form $uu$
for $u \in \{0,1,01,10\}$; see Pansiot~\cite{Pansiot}.

\begin{lemma}\label{TM:even}
Let $w$ be a factor of even length $n \ge 4$ of the Thue-Morse word
starting at an even position. Then all squares are at even positions,
and hence $M(w)=n/2-1$.	
\end{lemma}

We show in the next thorem that if an overlap-free word
satisfies the maximality condition
$2M(w)=|w|+3$ then $w$ has the form $w=auu$ (or symmetrically
$uua$). In Theorem~\ref{main}  we then construct overlap-free words
for which $M(w)$ reaches the upper bound.

We call the words $001001$ and $110110$ \emph{prefix special}
and their reversed words $100100$ and $011011$ 
\emph{suffix special}. These special words can occur only 
as a prefix (resp., suffix) of an overlap-free word since they cannot
be extended to the left (resp., right) without violating overlap-freeness.

\begin{theorem}\label{upper}
For overlap-free binary words $w$, we have $M(w) \le \left\lceil\frac{|w|}{2}\right\rceil+1$, i.e.,
\begin{equation}\label{eq:M}
2M(w) \le |w|+3\,.
\end{equation}
\end{theorem}

\begin{proof}
A word of length $n$ has $n-1$ positions of which
$\left[ n/2 \right]$ are odd positions
and $\left[ n/2 \right]-1$ or $\left[ n/2 \right]$ even positions
depending if $n$ is even or odd. 
For the bound \eqref{eq:M}
we look for the positions where the consecutive 
positions $p$ and $p+1$ are centres of squares. 

Let $|w|=n$.
If $w$ has a special prefix $001001$ or $110110$ then 
the positions~$3$ and $4$ are centres of squares. 
Similarly if $w$ has a suffix $100100$ or $011011$ then
the positions $n-4$ and $n-3$ are centres of squares.

Suppose then that $5 \le p \le n-5$, and assume that  
there are minimal length, and thus unbordered, 
squares $u$ and $v$ at positions $p$ and $p+1$. 
By appealing to symmetry, we can assume that $|u|\ge |v|$.
To avoid the overlap $ccc$ (with $c \in \{0,1\}$) that happens when $|u|=1=|v|$, 
necessarily $u=au'b$ for different letters $a,b$.
Also, $|u| \ge 3$ since if $u=ab$ then $uu=abab$
yields that $v=ba$
and there is an overlap $ababa$ in $w$.

\smallskip

\textbf{(A)} 
Assume first that $uu$ is not a prefix of $w$. 

\smallskip

\textbf{Claim 1}.
The square $uu$ is a suffix of $w$.

\smallskip
Indeed, $\beta=auu$ is a factor of $w$ 
since $buu$ would be an overlapping factor.
Hence $ab$ is a prefix of $u$. 
 
If $u=abb$ then $\beta=aabbabb$ ends in a special word. 
In this case there does not exist a square at $p+1$. 
(There is one at $p-1$.) Therefore $|u| \ge 4$. 

In order for $u$ to avoid $ab$ as a border, 
$bb$ must be a suffix of $u$. Hence 
$\beta=a(ab x abb)(ab x abb)$ for some $x$. 
It follows that $\beta$ is a suffix of $w$ since 
$\beta a$ and $\beta b$ have overlapping factors, 
$(ab x abb)(ab x abb)a$ and $bbb$, respectively.
This proves Claim 1.

\medskip
\textbf{Claim 2}.
$w=\beta$\,.

\smallskip
We have $x \ne \eps$ since if $x=\eps$, then $\beta=aababb ababb$ has an overlapping factor $babab$. 
Then $x$ has a prefix~$a$ 
in order for $u$ to avoid $abb$ as a border.
So far we have $\beta = a(aba y abb)(aba y abb)$.
In particular, $vv=baba$ is the minimal square at $p+1$.  
Now, either $y = \eps$ or, in order to avoid the factor
$babab$,  $ab$ is a prefix of $y$, i.e., $\beta= a(abaab z abb)(abaab z abb)$ for a word $z$.
In both cases, the prefix is $aabaab$ that can occur 
only at the beginning of $w$.
Hence $\beta=w$, and therefore
\[
w = auu = a(abaab z abb)(abaab z abb)\,.
\]

\smallskip 
Let $|u|=2k$ or $|u|=2k+1$.

The positions $1$ and $2$ of $u$ are not centres of squares in $u$,
and hence, by itself, $u$ can have at most $k-1$
odd or even centres of squares.
However, the position $3$ of $w$ is a centre of  $aabaab$
and the position $p+1$ of $w$ is centre of $vv$.
Together with the centres of $uu$ and $aa$,
we have $M(w) \le 2k+2$, and so $2M(w)\le |w|+3$
where $|w|=4k+1$.

\medskip

\textbf{(B)}
The case  where $uu$ is not a suffix of $w$ is symmetric to the case (A). In this case, $w$ has the form $uua$.

\smallskip
\textbf{(C)}
If $w=uu$ then $w$ cannot have special words in the prefix or 
in the suffix. 
Hence $M(w) \le |w|/2 +1$, i.e., $2M(w) \le |w|+2$. 

\smallskip

Finally, suppose that $w$ does not have centres of squares $p$ for
$5 \le p \le |w|-5$. Then the consecutive centres of squares
are in the special words.  E.g., by considering a possible factorisation
$w=(00100) 11 x 00 (11011)$ with special words at the ends, we obtain
$M(w) \le (|w|-10)/2 +6=|w|/2 +1$, and so $2M(w) \le |w|+2$.
\end{proof}

By the proof of Theorem~\ref{upper}, the bound 
$\left\lceil\frac{|w|}{2}\right\rceil+1$ can be obtained only
by overlap-free words of odd length.

\begin{theorem}\label{main}
There are infinitely many overlap-free binary words $w$ such that
\[
M(w) = \left\lceil\frac{|w|}{2}\right\rceil+1.
\]	
\end{theorem}

\begin{proof}
The the upper bound for $M(w)$ is given in Lemma~\ref{upper}.
For equality we rely on the factors of the Thue-Morse word $\tm$. Let $\alpha_n$ denote the factor of $\tm$ of length $3\sdot 2^n$ 
that starts at position~5, e.g.,
\[
\alpha_1 = 100110, \ \text{ and } \  
\alpha_2 = 1001 1001 0110\,.
\]
Then $|\mu(\alpha_n)|=2|\alpha_n|= |\alpha_{n+1}|$,
and $\mu(\alpha_n)=1001\sdot 0110 \cdots 0110\sdot 1001$,
where $\mu(\alpha_n)$ starts after $\mu(0110)$ at position~9 of $\tm$. Hence
$\alpha_{n+1} = 1001\sdot 1001 \cdots 0110$ is a conjugate
of $\mu(\alpha_n)$.
In particular, as a factor of $\tm$, $1001 \mu(\alpha_n)$ is overlap-free for all $n$.

We show inductively that $\alpha_n\alpha_n$ is overlap-free for all $n$. 
To begin with
$\alpha_1\alpha_1= 100110 100110$  is overlap-free.
For $n \ge 1$, we have
\begin{align*}
\alpha_{n+1}\alpha_{n+1} &= 
1001\,\mu(\alpha_n)(1001)^{-1}
(1001)\, \mu(\alpha_n)(1001)^{-1} \\
&=1001\,\mu(\alpha_n)\mu(\alpha_n)(1001)^{-1}\\
&=1001\,\mu(\alpha_n\alpha_n)(1001)^{-1}\,,
\end{align*}
where $1001\,\mu(\alpha_n)$ is overlap-free and, by the induction hypothesis,  so is  $\mu(\alpha_n)\mu(\alpha_n)$
since $\mu$ preserves overlap-freeness.
Lemma~\ref{squares} applied to
$1001\,\mu(\alpha_n\alpha_n)$ gives that
$1001\,\mu(\alpha_n\alpha_n)(1001)^{-1}$ is overlap-free.
Therefore $\alpha_{n+1}\alpha_{n+1}$ is overlap-free.

\smallskip
Denote by $\overline{\alpha}$ the prefix of $\alpha$ of 
length $|\alpha|-1$, i.e., truncate the last letter of $\alpha$.
Let
\[
w_n = 00\alpha_n \overline{\alpha}_n 
= 0(0\overline{\alpha}_n)(0\overline{\alpha}_n)
\] 
By Lemma~\ref{squares}, $0\alpha_n\overline{\alpha}_n$ is
overlap-free.
Also, $00\alpha_n \overline{\alpha}_n$ has a special prefix $001001$, 
and hence $w_n$ is overlap-free.

By Lemma~\ref{TM:even}, 
$M(\alpha_n\alpha_n)= (2|\alpha_n|/2 -1) +1=|\alpha_n|-1$
where the $+1$ comes from the middle position.
The centres of squares have even positions in
$\alpha_n\alpha_n$. 
Therefore also $M(\alpha_n\overline{\alpha}_n)=|\alpha_n|-1$.
The squares $001001$ and $00$ are at  the consecutive positions
$3$ and $4$ of $w_n$, and
in the middle portion we have $(0\overline{\alpha}_n)(0\overline{\alpha}_n)$ 
next to the squares $00$ and $1010$ (of even positions).
Since $|w_n|=2|\alpha_n| +1$, we have
\[
M(w_n)=|\alpha_n|-1+3 \ \text{ and so } \ 2 M(w_n)= |w_n|+3
\]
as required.
\end{proof}

\bibliographystyle{plain}
\bibliography{bibo.bib}

\begin{thebibliography}{1}

\bibitem{BerstelSeebold}
J.~{Berstel} and P.~{S\'e\'ebold}.
\newblock {A characterization of overlap-free morphisms}.
\newblock {\em {Discrete Appl. Math.}}, 46(3):275--281, 1993.

\bibitem{HarjuKarki}
T.~{Harju} and T.~{K\"arki}.
\newblock {On the number of frames in binary words}.
\newblock {\em {Theor. Comput. Sci.}}, 412(39):5276--5284, 2011.

\bibitem{LothaireI}
M.~Lothaire.
\newblock {\em Combinatorics on~Words}, volume~17 of {\em Encyclopedia
  of~Mathematics}.
\newblock Addison-Wesley, Reading, Massachusetts, 1983.

\bibitem{Pansiot}
J.~J. {Pansiot}.
\newblock {The Morse sequence and iterated morphisms}.
\newblock {\em {Inf. Process. Lett.}}, 12:68--70, 1981.

\bibitem{Thue:06}
A.~Thue.
\newblock {\"U}ber unendliche {Z}eichenreihen.
\newblock {\em Norske Vid. Selsk. Skr., I Mat.-nat. Kl. Christiania}, 7:1--22,
  1906.

\bibitem{Thue:12}
A.~Thue.
\newblock {\"U}ber die gegenseitige {L}age gleicher {T}eile gewisser
  {Z}eichenreihen.
\newblock {\em Norske Vid. Selsk. Skr., I Mat.-nat. Kl. Christiania}, 1:1--67,
  1912.

\end{thebibliography}

\end{document}